\pdfoutput=1
\RequirePackage{ifpdf}
\ifpdf 
\documentclass[pdftex]{sigma}
\else
\documentclass{sigma}
\fi

\numberwithin{equation}{section}

\newtheorem{Theorem}{Theorem}[section]
\newtheorem*{Theorem*}{Theorem}
\newtheorem{Corollary}[Theorem]{Corollary}

\newtheorem{Proposition}[Theorem]{Proposition}
 { \theoremstyle{definition}

\newtheorem{Remark}[Theorem]{Remark} }

\begin{document}
\allowdisplaybreaks

\renewcommand{\PaperNumber}{087}

\FirstPageHeading

\ShortArticleName{Bilateral Two-Parameter Mock Theta Functions and Related Applications}

\ArticleName{Bilateral Two-Parameter Mock Theta Functions\\ and Related Applications}

\Author{Chun WANG}

\AuthorNameForHeading{C.~Wang}

\Address{Department of Mathematics, Shanghai Normal University, Shanghai, 200234, P.R.~China}
\Email{\href{wangchun@shnu.edu.cn}{wangchun@shnu.edu.cn}}

\ArticleDates{Received January 20, 2025, in final form October 06, 2025; Published online October 17, 2025}

\Abstract{In this paper, we investigate new relationships for bilateral series related to two-parameter mock theta functions, which lead to many identities concerning the bilateral mock theta functions. In addition, interesting relations between the classical mock theta functions and the bilateral series are also concluded.}


\Classification{33D15; 11F27}

\section{Introduction}
Throughout the paper, we adopt the standard $q$-series notation in \cite{GR}.
For two complex variables~$a$ and $q$ with $|q|<1$ and $n\in \mathbb{Z}\cup \{\infty\}$, the shifted factorials of $a$ with base $q$ are defined~by
\begin{align*}
(a;q)_{\infty}:=\prod_{j=0}^{\infty}\bigl(1-aq^{j}\bigr) \qquad \text{and} \qquad
(a;q)_{n}:=\frac{(a;q)_{\infty}}{(aq^n;q)_{\infty}},
\end{align*}
specifically,
\begin{align*}
(a;q)_n:=
      \begin{cases}
       1, & \text{for $n=0$,} \\
       (1-a)(1-aq)\cdots\bigl(1-aq^{n-1}\bigr), & \text{for $n>0$,} \\
      1/\bigl(1-aq^{-1}\bigr)\bigl(1-aq^{-2}\bigr)\cdots\bigl(1-aq^n\bigr), & \text{for $n<0$.}
      \end{cases}
\end{align*}
For convenience, we use the shorthand notation
\begin{align*}
(a_1,a_2,\ldots,a_m;q)_n:=(a_1;q)_n(a_2;q)_n\cdots(a_m;q)_n,
\end{align*}
where $m$ is a positive integer.

The bilateral basic hypergeometric series is defined as follows.
\begin{align}\label{psi}
_r\psi_s\left[{a_1,a_2,\dots,a_r\atop b_1,b_2,\dots,b_s};q,z\right]
:=\sum_{n=-\infty}^\infty\frac{(a_1,a_2,\dots,a_r;q )_n}
{(b_1,b_2,\dots,b_s;q)_n}{\big[(-1)^{n}q^{n\choose2}\big]}^{s-r}z^n.
\end{align}
The series \eqref{psi} diverges for $z=0$ if $s < r$ and converges for $|b_1\cdots b_s/a_1\cdots a_r|<
|z| < 1$ if $r = s$.

One of Ramanujan's famous discoveries is the $_1\psi_1$ summation \cite[equation (5.2.1)]{GR}
\begin{align}\label{0-5}
{_1\psi_1}\left[{a \atop b};q,z\right]
=\frac{(q,b/a,az,q/az;q)_{\infty}}{(b,q/a,z,b/az;q)_{\infty}},
\end{align}
where $|b/a| < |z| < 1$.  The $_1\psi_1$ summation plays
important roles in $q$-series and mock modular forms (see \cite{AB,B,K}).

Taking $z\mapsto z/a$ and then $a\rightarrow\infty$, $b=0$ in \eqref{0-5} yields the well-known Jacobi's triple product identity $j(z;q)$, namely,
\begin{align*}
j(z;q):=(z,q/z,q;q)_{\infty}=\sum_{n=-\infty}^{\infty}(-1)^nq^{n\choose2}z^n.
\end{align*}
For brevity, we define $J_m:=j\bigl(q^m;q^{3m}\bigr)=(q^m;q^m)_\infty$.

In his last letter to Hardy, written three months before his death in 1920, Ramanujan \cite{R} introduced 17 functions which he called ``mock theta functions" and assigned them with the orders three, five, and seven. He did not explain precisely what he meant by a mock theta function, however, he referred to them as ``mock" because they imitated theta functions but did not fit into the established framework of modular forms. Here what Ramanujan referred to as ``theta" functions corresponds to what we now call weakly holomorphic modular forms. Ramanujan's
statements were interpreted by Andrews and Hickerson \cite{AH} to mean a mock theta function $f(q)$ is a function defined by a $q$-series which converges for $|q| < 1$ and which satisfies the following two conditions:
\begin{itemize}\itemsep=0pt
\item[(0)] For every root of unity $\zeta$, there is a theta function $\vartheta_\zeta(q)$ such that
 the difference $f(q)-\vartheta_\zeta(q)$ is bounded as ~$q\rightarrow\zeta$ radially.

\item[(1)] There is no single theta function which works for all $\zeta$: i.e., for every
 theta function $\theta(q)$ there is some root of unity $\zeta$ for which ~$f(q)-\vartheta_(q)$ is
 unbounded as $q\rightarrow\zeta$ radially.
\end{itemize}
A similar definition was given by Gordon and McIntosh \cite{GM1,GM}, where they
also distinguish between a mock theta function and a ``strong" mock theta function.
The modern view of mock theta functions is based on the work of Zwegers \cite{Z,Z1},
who showed that the mock theta functions are the holomorphic parts of certain weight 1/2 harmonic Maass forms and the order is related to the conductor of the Nebentypus of these modular objects. And also from Bringmann and Ono \cite{BO1,BO2}, we know that specializations of Appell--Lerch sums give rise to mock theta functions. The definition of Appell--Lerch sum is given by
\begin{align*}
m(x,q,z):=\frac{-z}{j(z;q)}\sum_{r=-\infty}^{\infty}
\frac{(-1)^{r}q^{r+1\choose 2}z^r}{1-q^rxz},
\end{align*}
where $x,z\in\mathbb{C}^{*}=\mathbb{C}-\{0\}$ with neither $z$ nor $xz$ an integral power of $q$.

During the past decades, the topic in terms of mock theta functions has attracted a lot of attention and has been studied widely by lots of mathematicians. For example, Andrews and Hickerson \cite{AH} proved eleven identities involving the sixth-order mock
theta functions which were listed in Ramanujan's lost notebook. Berndt and Chan
 \cite{BC} defined two new sixth-order mock theta functions and then provided four transformation formulas relating these two mock theta functions with Ramanujan's sixth-order mock theta functions. Gordon and McIntosh~\cite{GM2} constructed eight eighth-order mock theta
functions from the Rogers--Ramanujan type identities. Moreover, Ramanujan's lost notebook \cite{R} recorded ten identities, five for each of the two families, satisfied by the fifth-order mock theta functions, which have come to be known as the ``mock theta conjectures''. In \cite{AG}, Andrews and Garvan included that these identities in each family are equivalent.
Hickerson \cite{H} first confirmed these conjectures with the aid of the two-parameter mock theta function
\begin{align}\label{g3}
g_3(x,q):=\sum_{n=0}^{\infty}\frac{q^{n^2+n}}{(x,q/x;q)_{n+1}}.
\end{align}
In \cite{GM}, Gordon and McIntosh discussed how the mock theta functions with odd order can be expressed by $g_3(x,q)$,
and the even-order mock theta functions are related to $g_2(x,q)$ which is defined as
\begin{align*}
g_2(x,q):=\sum_{n=0}^{\infty}\frac{(-q;q)_nq^{n+1\choose2}}
{(x,q/x;q)_{n+1}}.
\end{align*}
Therefore, the functions $g_3(x,q)$ and $g_2(x,q)$ are called universal mock theta functions aptly. For more detailed introduction to the developments of mock theta functions, we refer the reader to see \cite{GM}.

The motivation for the present paper is the observation that
many of the mock theta functions are special cases of one ``side" ($n \geq 0$ or $n < 0$)
of certain general bilateral series. In \cite{W2}, Watson mentioned that the general term of the fifth mock theta function, when written in the following form
\[
f_{0}(q):=\sum_{n=0}^{\infty}\frac{q^{n^2}}{(-q;q)_n},
\]
exists for negative as well as for positive values of $n$; and this suggests the study of the complete series
\begin{align}\label{bf}
\sum_{n=-\infty}^{\infty}\frac{q^{n^2}}{(-q;q)_n}.
\end{align}
Obviously,
\begin{align*}
\sum_{n=-\infty}^{\infty}\frac{q^{n^2}}{(-q;q)_n}&=f_0(q)+\sum_{n=0}^{\infty}(-q;q)_{n-1}q^{n(n+1)/2}=f_0(q)+2\psi_{0}(q).
\end{align*}
He said the most curious feature of the complete series, however, is the phenomenon that, when we take the constituent parts, $f_{0}(q)$ and $2\psi_{0}(q)$ respectively, the odd parts of the constituents of them
 not only are expressible in terms of theta functions, but the latter is exactly double of the former.

In this paper, when we extend the defining series of a mock theta function to the complete series form as above, we have bilateral mock theta functions. We shall adopt the following notation for brevity: for a mock theta function
$N(q):=\sum_{n=0}^{\infty}a(n,q)$,
its bilateral series is defined by
$N_c(q):=\sum_{n=-\infty}^{\infty}a(n,q)$.
Thus, the bilateral fifth-order mock theta function defined in \eqref{bf} can be stated as $f_{0,c}(q)$.
 We also list the relation in terms of the third-order mock theta function $\omega(q)$ and its corresponding bilateral series $\omega_c(q)$ to make the notation more clear,
\begin{align}\label{0-1}
\omega_c(q)=\sum_{n=-\infty}^{\infty}\frac{q^{2n^2+2n}}{\bigl(q;q^2\bigr)_{n+1}^2}
=\sum_{n=0}^{\infty}\frac{q^{2n^2+2n}}{\bigl(q;q^2\bigr)_{n+1}^2}
+\sum_{n=0}^{\infty}\bigl(q;q^2\bigr)_n^2q^{2n}
=\omega(q)+\frac{J_1^2}{J_2^2}D_5(q),
\end{align}
where
\begin{align*}
D_5(q):=\frac{1}{\bigl(q;q^2\bigr)_{\infty}^2}\sum_{n=0}^{\infty}\bigl(q;q^2\bigr)_n^2q^{2n}
\end{align*}
is the mock theta functions researched by Hikami \cite{Hi}.

There also exist other kinds of bilateral series associated with mock theta functions. For instance, Choi \cite{Ch} investigated generalizations of classical mock theta functions by adding extra free parameters. Indeed, he developed the relations between the bilateral basic hypergeometric series and mock theta functions. Bajpai et al.\ \cite{BKLMR} and Mortenson \cite{Mo} presented radial limits results by using bilateral series of mock theta functions. In \cite{Mc}, Mc~Laughlin derived various transformations and summation formulae for some mock theta functions and the corresponding bilateral series. In \cite{HSZ}, Hu et al.\ deduced the Appell--Lerch sums for bilateral series associated with the third-order mock theta functions. Wei et al.\ \cite[Corollaries~2.1 and 2.2]{WYH} concluded two identities of bilateral mock theta functions in view of the transformation lemma from Andrews~\cite{A}.\looseness=-1

Inspired by the aforementioned work, this paper investigates bilateral mock theta functions from the perspective of bilateral two-parameter mock theta functions. The main results will be presented in the subsequent section. The remainder of the paper is structured as follows: Section \ref{sec1} is devoted to proving our theorems by means of Appell--Lerch sums. Additionally, in Appendix \ref{appendixA}, two of our main theorems are supplemented with alternative proofs, more specifically, Theorem \ref{3+2} is reproven herein via bilateral basic hypergeometric series, while Theorem~\ref{5+1} is reproven using a transformation lemma due to Andrews \cite{A}. Moreover, in Appendix \ref{appendixB}, the identities contained in the derived corollaries are categorized according to the orders of classical mock theta functions.

\section{Main results}

This section presents several relationships pertaining to bilateral two-parameter mock theta functions, along with a demonstration of how identities involving bilateral mock theta functions and classical identities may be derived from these relationships. Such as, for the bilateral two-parameter mock theta functions $g_{3,c}(x,q)$, which is the complete form of $g_3(x,q)$ defined in \eqref{g3}, we have a direct relation:
\begin{Theorem}\label{bg3-1}
We have
\begin{align} \label{bg3}
j\bigl(-xq;q^2\bigr)g_{3,c}\bigl(xq,q^2\bigr)+j\bigl(xq;q^2\bigr)g_{3,c}\bigl(-xq,q^2\bigr)=\frac{4J_4^5}{J_2^2j\bigl(x^2q^2;q^4\bigr)}.
\end{align}
\end{Theorem}
Let $\omega$ denote a primitive cube root of 1. Letting $(q,x)\rightarrow(q,1), (q,\omega), \bigl(q^3,q\bigr)$ in \eqref{bg3} gives the following identities associated with the bilateral series in terms of the third-order mock theta functions $\omega(q)$, $\rho(q)$, $\sigma(q)$ and the sixth-order $\beta(q)$ defined as follows (cf.\ \cite{GM1,GM,R})
\begin{gather*}
\rho(q):=\sum_{n=0}^{\infty}\frac{\bigl(q;q^2\bigr)_{n+1}q^{2n^2+2n}}{\bigl(q^3;q^6\bigr)_{n+1}},\qquad
\sigma(q):=\sum_{n=1}^{\infty}\frac{q^{3n^2-3n}}{\bigl(-q,-q^2;q^3\bigr)_n},\\
 \beta(q):=\sum_{n=0}^{\infty}\frac{q^{3n^2+3n+1}}{\bigl(q,q^2;q^3\bigr)_{n+1}}.
\end{gather*}
\begin{Corollary} \label{coro1}
We have
\begin{gather*}
\omega_c(q)+\frac{J_1^4J_4^2}{J_2^6}\omega_c(-q)=4\frac{J_1^2J_4^8}{J_2^9},\qquad
\rho_c(q)+\frac{J_2^3J_3^2J_{12}}{J_1^2J_4J_6^3}\rho_c(-q)=4\frac{J_3J_4^2J_{12}^2}{J_1J_6^3},\\
\sigma_c(q)+\frac{J_2J_3^2}{qJ_1^2J_6}\beta_c(q)=4\frac{J_6^5}{J_1J_2J_3^2}.
\end{gather*}
\end{Corollary}

We also establish the relationships for bilateral series associated
with two-parameter mock theta functions involving $g_3(x,q)$ and the following (see \cite{K,McI1,Mo1}):
\begin{gather*}
\begin{split}
& R(x,q):=\sum_{n=0}^{\infty}\frac{q^{n^2}}{(xq,q/x;q)_n},\qquad
K(x,q):=\sum_{n=0}^{\infty}\frac{(-1)^n\bigl(q;q^2\bigr)_nq^{n^2}}
{\bigl(xq^2,q^2/x;q^2\bigr)_n},\\
& K_1(x,q):=\sum_{n=0}^{\infty}\frac{(-1)^n\bigl(q;q^2\bigr)_nq^{(n+1)^2}} {\bigl(xq,q/x;q^2\bigr)_{n+1}},\qquad
S_2(x,q):=(1+1/x)\sum_{n=0}^{\infty}\frac{(-q;q)_{2n}q^{n+1}}
{\bigl(xq,q/x;q^2\bigr)_{n+1}}.
\end{split}
\end{gather*}
The two-parameter mock theta functions also possess combinatorial interpretations. For instance, $R(x,q)$ is a two-variable generating function for the Dyson rank function
\begin{align*}
R(x,q)=\sum_{n=0}^{\infty}\sum_{m=-\infty}^{\infty}N(m,n)z^mq^n,
\end{align*}
where $N(m,n)$ denotes the number of partitions of $n$ with rank $m$ \cite{Dyson}.
Many striking identities in terms of classical mock theta functions and bilateral series can be derived subsequently. In particular, the mock theta functions with different orders are connected. Following are the related mock theta functions including those mentioned above.

Ramanujan's four third-order mock theta functions (cf.~Ramanujan \cite{R}) are the following:
\begin{align*}
f(q):=\sum_{n=0}^{\infty}\frac{q^{n^2}}{(-q;q)_n^2},\qquad
\phi(q):=\sum_{n=0}^{\infty}\frac{q^{n^2}}{\bigl(-q^2;q^2\bigr)_n},\\
\psi(q):=\sum_{n=1}^{\infty}\frac{q^{n^2}}{\bigl(q;q^2\bigr)_n},\qquad
\chi(q):=\sum_{n=0}^{\infty}\frac{(-q;q)_nq^{n^2}}{\bigl(-q^3;q^3\bigr)_n},
\end{align*}
and (cf.\ Gordon and McIntosh \cite{GM1})
\begin{gather*}
\omega(q):=\sum_{n=0}^{\infty}\frac{q^{2n^2+2n}}{\bigl(q;q^2\bigr)_{n+1}^2},\qquad
\nu(q):=\sum_{n=0}^{\infty}\frac{q^{n^2+n}}{\bigl(-q;q^2\bigr)_{n+1}},\qquad\rho(q):=\sum_{n=0}^{\infty}\frac{\bigl(q;q^2\bigr)_{n+1}q^{2n^2+2n}}{\bigl(q^3;q^6\bigr)_{n+1}},\\
\xi(q):=1+2\sum_{n=1}^{\infty}\frac{q^{6n^2-6n+1}}{\bigl(q,q^5;q^6\bigr)_n},\qquad
\sigma(q):=\sum_{n=1}^{\infty}\frac{q^{3n^2-3n}}{\bigl(-q,-q^2;q^3\bigr)_n}.
\end{gather*}
Note that $\omega(q)$, $\nu(q)$ and $\rho(q)$ appear in \cite{R}, were rediscovered by Watson in \cite{W1}.

The sixth-order mock theta functions (cf.\ Andrews and Hickerson \cite{AH}, Berndt and Chan \cite{BC}) are the following:
\begin{gather*}
\gamma(q):=\sum_{n=0}^{\infty}\frac{(q;q)_nq^{n^2}}{\bigl(q^3;q^3\bigr)_n},\qquad
\phi_-(q):=\sum_{n=1}^{\infty}\frac{(-q;q)_{2n-1}q^n}{\bigl(q;q^2\bigr)_n},\\
\beta(q):=\sum_{n=0}^{\infty}\frac{q^{3n^2+3n+1}}{\bigl(q,q^2;q^3\bigr)_{n+1}},\qquad
\varPhi(q):=\sum_{n=0}^{\infty}\frac{(-q;q)_{2n}q^{n+1}}{\bigl(q;q^2\bigr)_{n+1}^2}.
\end{gather*}
The eighth-order mock theta functions (cf.\ Gordon and McIntosh \cite{GM2})  are the following:
\begin{gather*}
U_0(q):=\sum_{n=0}^{\infty}\frac{\bigl(-q;q^2\bigr)_nq^{n^2}}{\bigl(-q^4;q^4\bigr)_n},\qquad
U_1(q):=\sum_{n=0}^{\infty}\frac{\bigl(-q;q^2\bigr)_nq^{(n+1)^2}}
{\bigl(-q^2;q^4\bigr)_{n+1}},\\
V_1(q):=\sum_{n=0}^{\infty}\frac{\bigl(-q^4;q^4\bigr)_n q^{2n^2+2n+1}}{\bigl(q;q^2\bigr)_{2n+2}}.
\end{gather*}
The second-order mock theta functions (cf.\ McIntosh \cite{McI}) are the following:
\begin{gather*}
A(q):=\sum_{n=0}^{\infty}\frac{\bigl(-q;q^2\bigr)_n q^{(n+1)^2}}{\bigl(q;q^2\bigr)_{n+1}^2},\qquad
B(q):=\sum_{n=0}^{\infty}\frac{\bigl(-q^2;q^2\bigr)_n q^{n^2+n}}{\bigl(q;q^2\bigr)_{n+1}^2},\\
\mu(q):=\sum_{n=0}^{\infty}\frac{(-1)^n\bigl(q;q^2\bigr)_nq^{n^2}}{\bigl(-q^2;q^2\bigr)_n^2}.
\end{gather*}
We also recall the mock theta function $R_2(q)$ due to Gu and Hao \cite{GH}
\begin{align*}
R_2(q):=\sum_{n=0}^{\infty}\frac{(-1)^n\bigl(q;q^2\bigr)_n(1+q)q^{n^2+2n}}
{\bigl(-q^2;q^2\bigr)_n\bigl(-q^2;q^2\bigr)_{n+1}},
\end{align*}
which holds the relation with the second-order mock theta function $\mu(q)$ \cite[equation~(1.7)]{GH},
\begin{align*}
\mu(q)+R_2(q)=2.
\end{align*}

We then establish the following identity for the complete series in terms of $R(x,q)$ and $g_3(x,q)$.
\begin{Theorem}\label{3+2}
We have
\begin{align}\label{1+2}
&j\bigl(xq;q^2\bigr)R_c\bigl(x,q^2\bigr)-\bigl(1-x^{-1}\bigr)qj\bigl(x;q^2\bigr)g_{3,c}\bigl(xq,q^2\bigr)
=\frac{(1-x)J_1^5}{J_2^2j(x;q)}.
\end{align}
\end{Theorem}
Obviously, specializing \eqref{1+2} with $x=1$ and then replacing $q^2$ by $q$, it is straightforward to obtain the known result
$
R(1,q)=\frac{1}{J_1}$.

Besides, we set $x=-1$, $i$ in \eqref{1+2} to get the following results associated with the third-order mock theta functions $\omega(q)$, $f(q)$, $\phi(q)$ and $\nu(q)$.

 \begin{Corollary}
We have
\begin{gather}
f_c\bigl(q^2\bigr)-4q\frac{J_1^2J_4^4}{J_2^6}\omega_c(-q)
=\frac{J_1^8J_4^2}{J_2^9},\label{3-1}\\
\phi_c\bigl(q^2\bigr)-2q\frac{J_2J_8^2}{J_4^3}\nu_c\bigl(q^2\bigr)
=\frac{J_1^4J_8}{J_2J_4^3}.\label{3-2}
\end{gather}
\end{Corollary}
Note that identity \eqref{3-1} is equivalent to \cite[equation (9)]{ZS} and identity \eqref{3-2} becomes \cite[Corollary 2.2]{HZ} by using the facts \cite{BKLMR,Mc}
\begin{gather*}
\phi_c(q)=\phi(q)+2\psi(q)=2\psi_c(q)=\frac{J_2^7}{J_1^3J_4^3},\qquad
\nu_c(q)=\nu(q)+\nu(-q)=\frac{2J_4^3}{J_2^2}.
\end{gather*}

As defined earlier, $\omega$ is a primitive cube root of unity. It is evident that
\begin{gather*}
(-\omega q,-q/\omega;q)_n=\frac{\bigl(-q^3;q^3\bigr)_n}{(-q;q)_n},\qquad
\bigl(-\omega q,-q/\omega;q^2\bigr)_n=\frac{\bigl(-q^3;q^6\bigr)_n}{\bigl(-q;q^2\bigr)_n}.
\end{gather*}
These identities still hold when $n\rightarrow \infty$. Putting $x=-\omega$ in \eqref{1+2} and utilizing the above facts, we arrive at the following corollary involving the third-order mock theta functions $\chi(q)$, $\rho(q)$.
\begin{Corollary}
We have
\begin{align}
\label{0+3}
\chi_c\bigl(q^2\bigr)-q\frac{J_2^3J_3J_{12}^2}{J_1J_4^2J_6^3}\rho_c(-q)
=&\frac{J_1^2J_3^2J_{12}}{J_4J_6^3}.
\end{align}
\end{Corollary}

Next, we obtain the relation for bilateral series associated with universal mock theta function~$g_3(x,q)$.
\begin{Theorem}\label{5+1}
We have
\begin{align}\label{5+2}
xj\bigl(xq;q^2\bigr)g_{3,c}\bigl(x,q^2\bigr)+x^{-1}qj\bigl(x;q^2\bigr)g_{3,c}\bigl(xq,q^2\bigr)
=\frac{J_1^5}{J_2^2j(x;q)}-\frac{J_2j(x;q)}{J_1}.
\end{align}
\end{Theorem}
Identity \eqref{5+2} also implies some interesting identities analogous to \eqref{3-1}--\eqref{0+3}. Here, comparing with the series forms of the third-order mock theta functions $f(q)$, $\xi(q)$, we first define
\begin{gather*}
M(q):=\sum_{n=0}^{\infty}\frac{q^{n^2+n}}{(-1,-q;q)_{n+1}}=g_3(-1,q),\qquad
N(q):=\sum_{n=0}^{\infty}\frac{q^{6n^2+6n+1}}{\bigl(q,q^5;q^6\bigr)_{n+1}}=qg_3\bigl(q,q^6\bigr).
\end{gather*}
Note that one can find another representation for $M(q)$ in \cite{SW}. Putting $(q,x)\rightarrow(q,-1),\allowbreak \bigl(q^3,-q^2\bigr), \bigl(q^3,q^2\bigr)$ in \eqref{5+2}, the following identities \eqref{5-1}--\eqref{5-3} related to the bilateral forms of $M(q)$, $N(q)$, third-order mock theta function $\omega(q)$, $\sigma(q)$, sixth-order $\beta(q)$ are concluded.
\begin{Corollary}
We have
\begin{gather}
M_c\bigl(q^2\bigr)+2q\frac{J_1^2J_4^4}{J_2^6}\omega_c(-q)
=2\frac{J_4^2}{J_2^2}-\frac{J_1^8J_4^2}{2J_2^9},\label{5-1}\\
q^2\sigma_c\bigl(q^2\bigr)-\frac{J_1J_4^2J_6^3}{J_2^3J_3J_{12}^2}N_c(-q)
=\frac{J_4J_6}{J_2J_{12}}-\frac{J_1^2J_3^2J_4}{J_2^3J_{12}},\label{5-2}\\
\beta_c\bigl(q^2\bigr)+\frac{J_2^2J_3}{J_1J_6^2}N_c(q)
=\frac{J_2J_3^6}{J_1^2J_6^4}-\frac{J_2}{J_6}.\label{5-3}
\end{gather}
\end{Corollary}
Moreover, examining \eqref{3-1} and \eqref{5-1} yields the following result.
\begin{Corollary}
We have
$
f_c(q)+2M_c(q)=4\frac{J_2^2}{J_1^2}$.
\end{Corollary}

Furthermore, the results concerning the bilateral series associated with other two-parameter mock theta functions are also concluded as follows.
\begin{Theorem}\label{0+5}
We have
\begin{gather}
\frac{x}{x-1}K_c(x,q)+\bigg(1-\frac{1}{xq}\bigg)K_{1,c}(xq,q)
=\frac{j\bigl(xq^2;q^2\bigr)}{J_1},\label{1+3}\\
\frac{J_1j\bigl(-xq;q^2\bigr)}{x-1}K_{c}(x,q)
+\frac{2qJ_4j\bigl(x;q^2\bigr)}{x}g_{2,c}\bigl(-xq,q^2\bigr)\nonumber\\
\qquad=\frac{2qJ_4^5j\bigl(x;q^2\bigr)}{xJ_2^2 j\bigl(x^2q^2;q^4\bigr)}
-\frac{J_4^{11}j\bigl(x;q^2\bigr)j\bigl(-x^2q^2;q^4\bigr)}
{x^2J_2^4J_8^4j\bigl(x^2q^2;q^4\bigr)j\bigl(-x^2q^4;q^4\bigr)}\nonumber\\
\phantom{\qquad=}{}+\frac{J_4^5j\bigl(x;q^2\bigr)j\bigl(-xq;q^2\bigr)}
{x^2J_8^2j\bigl(xq;q^2\bigr)j\bigl(-x^2q^4;q^4\bigr)}
-\frac{J_1^5j\bigl(-xq;q^2\bigr)}{J_2^2j(x;q)},\label{1+4}\\
2xJ_4^2j\bigl(xq;q^2\bigr)S_{2,c}\bigl(x^2,q^2\bigr)-qJ_2^2j\bigl(x^2q^2;q^4\bigr)g_{3,c}\bigl(xq,q^2\bigr)\nonumber\\
\qquad=\frac{xJ_4^{12}j\bigl(xq;q^2\bigr)j\bigl(-x^2q^2;q^4\bigr)}{J_2^4J_8^4j\bigl(x^2q^2;q^4\bigr)j\bigl(-x^2;q^4\bigr)}
-\frac{xJ_2^5j\bigl(x^2q^2;q^4\bigr)j\bigl(-xq;q^2\bigr)^2}{2J_4^2j\bigl(-x^2;q^2\bigr)j\bigl(xq;q^2\bigr)}\nonumber\\
\phantom{\qquad=}{}-\frac{xJ_2^5j\bigl(xq;q^2\bigr)j\bigl(x^2q^2;q^4\bigr)}{2J_4^2j\bigl(-x^2;q^2\bigr)}.\label{1+9}
\end{gather}
\end{Theorem}

Taking $x=-1$ in \eqref{1+3}, \eqref{1+4} and $x=1$ in \eqref{1+9} yields the following identities \eqref{0+1add}--\eqref{1+12}.
\begin{Corollary}
We have
\begin{gather}
\mu_c(q)+R_{2,c}(q)=4\frac{J_4^2}{J_1J_2},\label{0+1add}\\
\mu_c(q)+8q\frac{J_4^3}{J_1^3}B_c(q)=8q\frac{J_4^8}{J_1^3J_2^4}
+2\frac{J_4^{20}}{J_1^3J_2^8J_8^8}+\frac{J_1^5}{J_2^4}
-2\frac{J_1J_4^{10}}{J_2^6J_8^4},\label{0+4}\\
16\varPhi_c\bigl(q^2\bigr)-4q\frac{J_2^5}{J_1^2J_4^3}\omega_c(q)
=2\frac{J_4^{17}}{J_2^8J_8^8}-\frac{J_2^{20}}{J_1^8J_4^{11}}
-\frac{J_2^8}{J_4^7}.\label{1+12}
\end{gather}
\end{Corollary}

\begin{Remark}
Together the following identity due to Zhang and Song \cite{ZS}
\begin{align*}
\mu_c(q)+8q\frac{J_4^3}{J_1^3}B_c(q)=\frac{J_2^{20}}{J_1^{11}J_4^8},
\end{align*}
with \eqref{0+4} yields the following theta series identity:
\begin{align*}
\frac{J_2^{20}}{J_1^{11}J_4^8}=8q\frac{J_4^8}{J_1^3J_2^4}
+2\frac{J_4^{20}}{J_1^3J_2^8J_8^8}+\frac{J_1^5}{J_2^4}
-2\frac{J_1J_4^{10}}{J_2^6J_8^4}.
\end{align*}
The above result can be proved by using a \textsc{Maple} program \cite{FG}. In a similar manner, identity~\eqref{1+12} becomes
\begin{align*}
16\varPhi_c\bigl(q^2\bigr)-4q\frac{J_2^5}{J_1^2J_4^3}\omega_c(q)=-8q\frac{J_4^5}{J_2^4}.
\end{align*}
\end{Remark}

Finally, we shall establish the following series identities in terms of the third-order mock theta functions $\omega(q)$, $f(q)$, $\gamma(q)$, $\sigma(q)$, $\nu(q)$, $\phi(q)$, $\psi(q)$, the eighth-order $U_0(q)$, $U_1(q)$, $V_1(q)$, the second-order $A(q)$, $B(q)$, $\mu(q)$ and the sixth-order $\gamma(q)$, $\phi_-(q)$, $\beta(q)$, $\varPhi(q)$. Different from the above, they are the relations between the classical mock theta functions and the bilateral mock theta functions.
\begin{Theorem}\label{1+6}
We have
\begin{gather}
2\frac{J_1^2}{J_2^2}\varPhi(q)= q\omega_c(q),\label{1+13}\\
\frac{J_2J_3}{J_1J_6}\gamma(q)-3q\sigma_c(q)= \frac{J_1^4}{J_2^2J_3},\label{1+7}\\
\beta_c(q)-\frac{J_1}{J_3}\phi_-(q)= q\frac{J_1J_6^6}{J_2^2J_3^4},\label{1+8}\\
\frac{J_2^5}{J_1^3J_4^2}f\bigl(q^2\bigr)+4A_c(-q)= \frac{J_1^5}{J_2^4},\label{1+14}\\
\frac{J_4^2}{J_1J_8}\phi\bigl(q^2\bigr)+2U_{1,c}(-q)= \frac{J_1^3}{J_2J_4},\label{1+15}\\
\mu_c(q)-4q\frac{J_4^2}{J_1J_2}\omega(-q)= \frac{J_1^5}{J_2^4},\label{1+16}\\
U_{0,c}(-q)-2q\frac{J_2J_8}{J_1J_4}\nu\bigl(q^2\bigr)= \frac{J_1^3}{J_2J_4},\label{1+17}\\
8q\frac{J_4^3}{J_1^3}B_c(q)+4q\frac{J_4^2}{J_1J_2}\omega(-q)= 2\frac{J_4^{20}}{J_1^3J_2^8J_8^8}
+8q\frac{J_4^8}{J_1^3J_2^4}-2\frac{J_1J_4^{10}}{J_2^6J_8^4},\label{1+18}\\
2V_{1,c}(q)-\frac{J_1J_4}{J_2J_8}\psi(-q)= \frac{J_8^9}{J_2J_4^3J_{16}^4}+2q\frac{J_8^3}{J_2J_4}-\frac{J_1^2J_8^4}{J_2^2J_4J_{16}^2}.\label{1+19}
\end{gather}
\end{Theorem}

\begin{Remark}
By means of \eqref{0-1}, one sees that identity \eqref{1+13} is equivalent to \cite[equation~(16)]{Hi}.
The difference between \eqref{0+4} and \eqref{1+16} yields \eqref{1+18}.
\end{Remark}

\section{Proofs of the main results} \label{sec1}
In this section, we shall give the concise proofs of Theorems \ref{bg3-1}, \ref{3+2}, \ref{5+1}, \ref{0+5}, and \ref{1+6} by using Appel--Lerch sums. First, we need some basic relations from \cite[equations~(4.1), (3.1), (5.12), (5.5), (5.18), (6.10)]{Mo}.
\begin{Proposition} We have
\begin{gather}
\label{2-5}
qg_{3,c}\bigl(xq,q^2\bigr)=-\frac{xj\bigl(xq;q^2\bigr)}{J_2}m\bigl(x^2,q^2,q/x\bigr),\\
R_c(x,q)=(1-x)\frac{j(x;q)}{J_1}m\bigl(x^2,q,1/x\bigr),\label{2.2-add}\\
2S_{2,c}(x,q)=-\frac{J_1j\bigl(xq;q^2\bigr)}{J_2^2}m(x,q,-1)
-\frac{J_1^5j\bigl(xq;q^2\bigr)}{2J_2^4j(-x;q)}+\frac{J_2^{10}j\bigl(-xq;q^2\bigr)}{J_1^4J_4^4j \bigl(xq;q^2\bigr)j\bigl(-x;q^2\bigr)}, \label{3-3}
\end{gather}
and
\begin{gather}\label{2-7}
(1-1/x)K_{1,c}(x,q)=\frac{j\bigl(xq;q^2\bigr)}{J_1}\bigl(1+xg_3\bigl(x,q^2\bigr)\bigr)
-\frac{J_1^4}{J_2^2j(x;q)},\\
\frac{1}{1-x}K_{c}(x,q)=-\frac{qj\bigl(x;q^2\bigr)}{xJ_1}g_3\bigl(xq,q^2\bigr)
+\frac{J_1^4}{J_2^2j(x;q)},\label{2-4}\\
g_{2,c}(x,q)=-\frac{j(x;q)}{2J_2}g_3(-x,q)
+\frac{J_2^{10}j\bigl(-x^2;q^2\bigr)}{2xJ_1^4J_4^4j\bigl(x^2;q^2\bigr)j\bigl(-x^2q;q^2\bigr)}\nonumber\\
\phantom{g_{2,c}(x,q)=}{}+\frac{J_2^4}{J_1^2j\bigl(x^2;q^2\bigr)}
-\frac{J_2^4j(x;q)}{2xJ_4^2j(-x;q)j\bigl(-x^2q;q^2\bigr)},\label{2-8}
\end{gather}
where
 \[g_3(x,q)=-x^{-1}m\bigl(q^2/x^3,q^3,x^2\bigr)-x^{-2}m\bigl(q/x^3,q^3,x^2\bigr).\]
\end{Proposition}

We now ready to prove our theorems.

\begin{proof} [Proof of Theorem \ref{bg3-1}]
We first recall the relation follows from \cite[equation~(3.7)]{HM} (or see \cite[Chapter~8]{BFOR} and~\cite{Z})
\begin{align}\label{2-6}
m(x,q,z)-m(x,q,z_0)=\frac{z_0J_1^3j(z/z_0;q)j(xzz_0;q)}
{j(z_0;q)j(z;q)j(xz_0;q)j(xz;q)}.
\end{align}

In view of \eqref{2-5} and \eqref{2-6}, we find that
\begin{gather*}
j\bigl(-xq;q^2\bigr)g_{3,c}\bigl(xq,q^2\bigr)+j\bigl(xq;q^2\bigr)g_{3,c}\bigl(-xq,q^2\bigr)\\
\qquad=-\frac{xj\bigl(xq;q^2\bigr)j\bigl(-xq;q^2\bigr)}{qJ_2}\bigl(m\bigl(x^2,q^2,q/x\bigr)-m\bigl(x^2,q^2,-q/x\bigr)\bigr)=\frac{4J_4^5}{J_2^2j\bigl(x^2q^2;q^4\bigr)},
\end{gather*}
which equals to \eqref{bg3}.
\end{proof}
\begin{proof}[Proof of Theorem \ref{3+2}]
Combining \eqref{2-5} and \eqref{2.2-add},
we have
\begin{gather*}
\bigl(q/x,xq;q^2\bigr)_{\infty}R_c\bigl(x,q^2\bigr)
-\bigl(1/x,x;q^2\bigr)_{\infty}qg_{3,c}\bigl(xq,q^2\bigr)\\
\qquad=\bigl(q/x,xq;q^2\bigr)_{\infty}(1-x)\frac{j\bigl(x;q^2\bigr)}{J_2}m\bigl(x^2,q^2,1/x\bigr)\\
\phantom{\qquad=}{}+\bigl(1/x,x;q^2\bigr)_{\infty}\frac{x j\bigl(xq;q^2\bigr)}{J_2}m\bigl(x^2,q^2,q/x\bigr)\\
\qquad=(1-x)(q/x,x;q)_{\infty}m\bigl(x^2,q^2,1/x\bigr)+x(1/x,x;q)_{\infty}m\bigl(x^2,q^2,1/x\bigr)\\
\phantom{\qquad=}{}+\frac{(1/x,x;q)_{\infty}J_2^3j\bigl(q;q^2\bigr)j\bigl(q;q^2\bigr)}
{j\bigl(1/x;q^2\bigr)j\bigl(q/x;q^2\bigr)j\bigl(xq;q^2\bigr)j\bigl(x;q^2\bigr)}\\
\qquad=\frac{\bigl(q;q^2\bigr)_{\infty}^4\bigl(q^2;q^2\bigr)_{\infty}}{(q/x,xq;q)_{\infty}},
\end{gather*}
where the second identity follows from \eqref{2-6}. This proves \eqref{1+2}.
\end{proof}
\begin{proof}[Proof of Theorem \ref{5+1}]
By \eqref{2-5} and \cite[(3.2c)]{HM},
$
m(qx,q,z)=1-xm(x,q,z)$,
and we have
\begin{gather}
xj\bigl(xq;q^2\bigr)g_c\bigl(x,q^2\bigr)+x^{-1}qj\bigl(x;q^2\bigr)g_c\bigl(xq,q^2\bigr)\nonumber\\
\qquad=\frac{J_2j(x;q)}{J_1}\bigl(m\bigl(x^2,q^2,q^2/x\bigr)-1-m\bigl(x^2,q^2,q/x\bigr)\bigr).\label{4-5}
\end{gather}
We further employ \eqref{2-6} to obtain
\begin{align*}
m\bigl(x^2,q^2,q^2/x\bigr)-m\bigl(x^2,q^2,q/x\bigr)=\frac{J_1^6}{J_2^3j(x;q)^2}.
\end{align*}
Inserting the above result into \eqref{4-5}, we get \eqref{5+2}.
\end{proof}

\begin{proof}[Proof of Theorem \ref{0+5}]

We use \eqref{2-7} as well as \eqref{2-4} to obtain \eqref{1+3}. Similarly, we employ~\eqref{2-4} and \eqref{2-8} to conclude \eqref{1+4}.

By \eqref{2-6}, we have
\begin{align*}
m\bigl(x^2,q^2,q/x\bigr)=m\bigl(x^2,q^2,-1\bigr)
-\frac{J_2^4j\bigl(-xq;q^2\bigr)^2}{2J_4^2j\bigl(-x^2;q^2\bigr)j\bigl(xq;q^2\bigr)^2}.
\end{align*}
Combining \eqref{2-5}, \eqref{3-3} and the above result yields \eqref{1+9}.
\end{proof}

\begin{proof}[Proof of Theorem \ref{1+6}.] Putting $x=1$ in \eqref{2-5}, we deduce
\smash{$
q\omega_c(q)=-\frac{J_1^2}{J_2^2}m\bigl(1,q^2,q\bigr)$}.
In view of \cite[equation~(2.12)]{Mo1}
$
S_2(x,q)=-m\bigl(x,q^2,q\bigr)$,
one sees that
$
2\varPhi(q)=-m\bigl(1,q^2,q\bigr)$.
The result~\eqref{1+13} follows readily.

Using \eqref{2-5}, we obtain
\begin{align}\label{4-1}
\sigma_c(q)=\frac{J_2J_3}{qJ_1J_6}m\bigl(q,q^3,-q\bigr).
\end{align}
Examining the above equality and \cite[equation~(5.29)]{HM}
\begin{align*}
\gamma(q)=3m\bigl(q,q^3,-q\bigr)+\frac{J_1^5J_6}{J_2^3J_3^2},
\end{align*}
we conclude \eqref{1+7}. By proceeding as above, identity \eqref{1+8} follows from
\begin{align}\label{4-2}
\beta_c(q)=-\frac{J_1}{J_3}m\bigl(q,q^3,q\bigr),
\end{align}
and the fact \cite[equation~(5.30)]{HM}
\begin{align*}
\phi_-(q)=-m\bigl(q,q^3,q\bigr)-q\frac{J_6^6}{J_2^2J_3^3}.
\end{align*}

Letting $x=-1$ in \eqref{2-7} gives
\begin{align*}
2K_{1,c}(-1,q)=\frac{J_2^5}{J_1^3J_4^2}\bigl(1-g_3\bigl(-1,q^2\bigr)\bigr)-\frac{J_1^5}{2J_2^4}.
\end{align*}
Owing to the fact that
$
 K_{1,c}(-1,q)=-A_c(-q)$,
$f(q)= 2-2g_3(-1,q)$,
we have
\begin{align*}
-4A_c(-q)=\frac{J_2^5}{J_1^3J_4^2}f\bigl(q^2\bigr)-\frac{J_1^5}{J_2^4}.
\end{align*}
The result \eqref{1+14} follows readily. Similarly, we set $x=i$ in \eqref{2-7} to obtain
\begin{align*}
(1+i)K_{1,c}(i,q)=\frac{j\bigl(iq;q^2\bigr)}{J_1}\bigl(1+ig_3(i,q^2)\bigr)
-\frac{J_1^4}{J_2^2j(i;q)}.
\end{align*}
Using the fact that
$K_{1,c}(i,q)=-U_{1,c}(-q)$,
$\phi(q)= (1-i)\bigl(1+ig_3(i,q)\bigr)$,
yields \eqref{1+15} by performing further calculations.

In a similar manner, taking $x=-1$ and $i$ in \eqref{2-4} and after simplifying, we deduce \eqref{1+16} and \eqref{1+17}.
Replacing $q$ by $q^2$ and letting $x=q$ in \eqref{2-8}, we get \eqref{1+18}. Replacing $q$ by $q^4$ and letting $x=q$ in \eqref{2-8}, we have \eqref{1+19}.
\end{proof}

\begin{Remark}
We employ \eqref{2-6} to deduce
\begin{align*}
m\bigl(q,q^3,-q\bigr)=m\bigl(q,q^3,q\bigr)+\frac{4qJ_6^6}{J_2^2J_3^3}.
\end{align*}
Together the above equality with \eqref{4-1} as well as \eqref{4-2} yields the last identity in Corollary \ref{coro1}.
On the other hand, letting $x=1$ in \eqref{3-3}, we obtain
\begin{align*}
4\varPhi_c(q)=-\frac{J_1^3}{J_2^3}m(1,q,-1)+\frac{J_2^{17}}{2J_1^8J_4^8}
-\frac{J_1^8}{4J_2^7}.
\end{align*}
Combining the above identity with the following due to Hu et al.~\cite{HSZ}
\begin{align*}
f_c(q)=4\frac{J_2^2}{J_1^2}m(1,q,-1),
\end{align*}
we arrive at
\begin{align}\label{1+11}
16\varPhi_c(q)+\frac{J_1^5}{J_2^5}f_c(q)=2\frac{J_2^{17}}{J_1^8J_4^8}
-\frac{J_1^8}{J_2^7}.
\end{align}
It is evident that the difference between \eqref{1+12} and \eqref{1+11} yields \eqref{3-1}.
\end{Remark}

\appendix

\section[Proofs of Theorems 2.3 and 2.6]{Proofs of Theorems \ref{3+2} and \ref{5+1}}\label{appendixA}

In this appendix, we shall reprove Theorem \ref{3+2} by using the bilateral basic hypergeometric series and reprove Theorem \ref{5+1} based on a transformation lemma due to Andrews \cite{A}.

\subsection[Second proof of Theorem 2.3]{Second proof of Theorem \ref{3+2}} \label{sec2}

\begin{proof}
Recall that the transformation formula due to Bailey \cite[equation~(2.3)]{Ba}
\begin{align*}
_2\psi_2\left[{a,b \atop c,d};q,z\right]
=\frac{(az,d/a,c/b,dq/abz;q)_{\infty}}
{(z,d,q/b,cd/abz;q)_{\infty}}
{_2\psi_2}\left[{a,abz/d \atop az,c};q,d/a\right].
\end{align*}
We let $(q,a,b,c,d,z)\rightarrow\bigl(q^2,a^2,b^2,abq,-abq,-q^2/z\bigr)$ to obtain
\begin{gather*}
{_2\psi_2}\left[{a^2, b^2 \atop abq,-abq};q^2,-q^2/z\right]\\
\qquad=\frac{\bigl(-a^2q^2/z,-bq/a,aq/b,zq/ab;q^2\bigr)_{\infty}}
{\bigl(-q^2/z,-abq,q^2/b^2,z;q^2\bigr)_{\infty}}
{_2\psi_2}\left[{a^2,abq/z \atop -a^2q^2/z,abq};q^2,-bq/a\right].
\end{gather*}
Substituting the above identity into the following result from \cite[Theorem 2.3]{HZ}
\begin{gather}
\bigl(z;q^2\bigr)_{\infty}{_2\psi_2}\left[{a^2, b^2 \atop abq,-abq};q^2,-q^2/z\right]\nonumber\\
\qquad=\frac{\bigl(a^2q^2,b^2q^2;q^4\bigr)_{\infty}}{\bigl(a^2b^2q^2,q^2;q^4\bigr)_{\infty}}
\sum_{n=-\infty}^{\infty}\frac{q^{4n^2}(z/ab)^{2n}}
{\bigl(q^4/a^2,q^4/b^2;q^4\bigr)_n}\nonumber\\
\phantom{\qquad=}{}-\frac{q^2}{z}\frac{\bigl(a^2,b^2;q^4\bigr)_{\infty}}{\bigl(a^2b^2q^2,q^2;q^4\bigr)_{\infty}}
\sum_{n=-\infty}^{\infty}\frac{q^{4n^2-4n}(z/ab)^{2n}}
{\bigl(q^2/a^2,q^2/b^2;q^4\bigr)_n},\label{2-2}
\end{gather}
we find that
\begin{gather}
\frac{\bigl(-a^2q^2/z,-bq/a,aq/b,zq/ab;q^2\bigr)_{\infty}}
{\bigl(-q^2/z,-abq,q^2/b^2;q^2\bigr)_{\infty}}
{_2\psi_2}\left[{a^2,abq/z \atop -a^2q^2/z,abq};q^2,-bq/a\right]\nonumber\\
\qquad=\frac{\bigl(a^2q^2,b^2q^2;q^4\bigr)_{\infty}}{\bigl(a^2b^2q^2,q^2;q^4\bigr)_{\infty}}
\sum_{n=-\infty}^{\infty}\frac{q^{4n^2}(z/ab)^{2n}}
{\bigl(q^4/a^2,q^4/b^2;q^4\bigr)_n}\nonumber\\
\phantom{\qquad=}{}-\frac{q^2}{z}\frac{\bigl(a^2,b^2;q^4\bigr)_{\infty}}{\bigl(a^2b^2q^2,q^2;q^4\bigr)_{\infty}}
\sum_{n=-\infty}^{\infty}\frac{q^{4n^2-4n}(z/ab)^{2n}}
{\bigl(q^2/a^2,q^2/b^2;q^4\bigr)_n}.\label{2-3}
\end{gather}
Putting $z=1$, $a^2=1/x$ and $b^2=x$ in \eqref{2-3} gives
\begin{gather}
\frac{\bigl(-q^2/x,-xq,q/x,q;q^2\bigr)_{\infty}}{\bigl(-q^2,-q,q^2/x;q^2\bigr)_{\infty}}
{_1\psi_1}\left[{1/x \atop -q^2/x};q^2,-xq\right]\nonumber\\
\qquad=\frac{\bigl(q^2/x,xq^2;q^4\bigr)_{\infty}}{\bigl(q^2;q^4\bigr)_{\infty}^2}
\sum_{n=-\infty}^{\infty}\frac{q^{4n^2}}{\bigl(xq^4,q^4/x;q^4\bigr)_n}\nonumber\\
\phantom{\qquad=}{}-\frac{q^2\bigl(1/x,x;q^4\bigr)_{\infty}}{\bigl(q^2;q^4\bigr)_{\infty}^2}
\sum_{n=-\infty}^{\infty}\frac{q^{4n^2-4n}}{\bigl(xq^2,q^2/x;q^4\bigr)_n}.\label{2-1}
\end{gather}
By means of \eqref{0-5}, we rewrite \eqref{2-1} as
\begin{align*}
\frac{\bigl(q^2;q^2\bigr)_{\infty}\bigl(q^2;q^4\bigr)_{\infty}}{\bigl(q^2/x,xq^2;q^2\bigr)_{\infty}}
=\frac{\bigl(q^2/x,xq^2;q^4\bigr)_{\infty}}{\bigl(q^2;q^4\bigr)_{\infty}^2}R_c\bigl(x,q^4\bigr)
-\frac{q^2\bigl(1/x,x;q^4\bigr)_{\infty}}{\bigl(q^2;q^4\bigr)_{\infty}^2}g_{3,c}\bigl(xq^2,q^4\bigr).
\end{align*}
We further replace $q^2$ by $q$ to obtain \eqref{1+2}.
\end{proof}

\begin{Remark}
It needs to be pointed out that when $z=1$, the left hand side of \eqref{2-2} is not equal to $0$, actually, it equals to the corresponding theta series from the perspective of \eqref{2-3}. In light of this, we list identity \eqref{3-1} as the revised version of \cite[Corollary 2.4]{HZ}. Likewise, one can discuss the identities in \cite[Corollaries 2.6--2.9]{HZ} and derive the right results as that of \eqref{3-1}, we omit the details here.
\end{Remark}

\subsection[Second proof of Theorem 2.6]{Second proof of Theorem \ref{5+1}} \label{sec3}
\begin{proof}
 Recall the following transformation lemma from Andrews \cite{A}: subject to suitable convergence conditions, if
\begin{align}\label{1-1}
c_{n}=\sum_{m=0}^{\infty}a_{m+n}b_{m},
\end{align}
then
\begin{align}\label{1-2}
\sum_{m=0}^{\infty}b_{m}\sum_{n=-\infty}^{\infty}a_{n}
=\sum_{n=-\infty}^{\infty}c_{n}.
\end{align}
We set
\begin{align*}
a_n=\frac{z^nq^{\frac{n^2}{2}}}{(c;q)_n} \qquad \text{and} \qquad
b_m=\frac{c^m(a,1/a;q)_m}{\bigl(zq^{\frac{1}{2}}\bigr)^m(q,-q;q)_m}
\end{align*}
in \eqref{1-1} to get
\begin{align*}
c_n={}&\sum_{m=0}^{\infty}a_{n+m}b_m
=\sum_{m=0}^{\infty}\frac{z^{n+m}q^{\frac{(n+m)^2}{2}}c^m(a,1/a;q)_m}
{(c;q)_{n+m}\bigl(zq^{\frac{1}{2}}\bigr)^m(q,-q;q)_m}\\
={}&\frac{z^nq^{\frac{n^2}{2}}}{(c;q)_n}\sum_{m=0}^{\infty}
\frac{(a,1/a;q)_mq^{m\choose2}(cq^n)^m}
{(q,-q,cq^n;q)_m}\\
={}&\frac{z^nq^{\frac{n^2}{2}}}{(c;q)_n}
\bigg\{\frac{\bigl(acq^{n+1},cq^n/a;q^2\bigr)_{\infty}
+a\bigl(acq^n,cq^{n+1}/a;q^2\bigr)_{\infty}}{(1+a)(cq^n;q)_{\infty}}\bigg\}\\
={}&\frac{z^nq^{\frac{n^2}{2}}}{(1+a)(c;q)_{\infty}}
\big\{\bigl(acq^{n+1},cq^n/a;q^2\bigr)_{\infty}
+a\bigl(acq^n,cq^{n+1}/a;q^2\bigr)_{\infty}\big\},
\end{align*}
where the penultimate identity follows from the equality \cite[Example~8]{ZS2}
\begin{align*}
\sum_{n=0}^{\infty}\frac{(a,1/a;q)_nq^{n\choose2}c^n}{\bigl(q^2;q^2\bigr)_n(c;q)_n}
=\frac{\bigl(qac,c/a;q^2\bigr)_{\infty}+a\bigl(ac,qc/a;q^2\bigr)_{\infty}}{(1+a)(c;q)_{\infty}}.
\end{align*}
We then employ \eqref{0-5} to get
\begin{align*}
\sum_{n=-\infty}^{\infty}a_n=\frac{\bigl(q,-zq^{\frac{1}{2}},-q^{\frac{1}{2}}/z;q\bigr)_{\infty}}
{\bigl(c,-c/zq^{\frac{1}{2}};q\bigr)_{\infty}}.
\end{align*}
With the aid of \eqref{1-2}, it follows that
\begin{gather*}
\frac{1}{1+a}\sum_{n=-\infty}^{\infty}
\bigl(acq^{n+1},cq^n/a;q^2\bigr)_{\infty}z^nq^{\frac{n^2}{2}}+\frac{a}{1+a}\sum_{n=-\infty}^{\infty}
\bigl(acq^n,cq^{n+1}/a;q^2\bigr)_{\infty}z^nq^{\frac{n^2}{2}}\\
\qquad=\frac{\bigl(q,-zq^{\frac{1}{2}},-q^{\frac{1}{2}}/z;q\bigr)_{\infty}}{\bigl(-c/zq^{\frac{1}{2}};q\bigr)_{\infty}}
\sum_{m=0}^{\infty}\frac{(a,1/a;q)_m\bigl(c/zq^{\frac{1}{2}}\bigr)^m}{(q,-q;q)_m}.
\end{gather*}
Decomposing the sum according to the parity of $n$,
and then setting $a=x/q^{\frac{1}{2}}$, $c=q^{\frac{3}{2}}$, $z=-1$, we are led to
\begin{gather*}
\frac{\bigl(xq^2,q^2/x;q^2\bigr)_{\infty}}{1+x/q^{\frac{1}{2}}}\sum_{n=-\infty}^{\infty}
\frac{q^{2n^2}}{\bigl(xq^2,q^2/x;q^2\bigr)_n}-\frac{q^{\frac{1}{2}}\bigl(xq^3,q^3/x;q^2\bigr)_{\infty}}{1+x/q^{\frac{1}{2}}}\sum_{n=-\infty}^{\infty}
\frac{q^{2n^2+2n}}{\bigl(xq^3,q^3/x;q^2\bigr)_n}\\
\qquad+\frac{x\bigl(xq,q^3/x;q^2\bigr)_{\infty}}{q^{\frac{1}{2}}\bigl(1+x/q^{\frac{1}{2}}\bigr)}\sum_{n=-\infty}^{\infty}
\frac{q^{2n^2}}{\bigl(xq,q^3/x;q^2\bigr)_n}-\frac{x\bigl(xq^2,q^4/x;q^2\bigr)_{\infty}}{1+x/q^{\frac{1}{2}}}\sum_{n=-\infty}^{\infty}
\frac{q^{2n^2+2n}}{\bigl(xq^2,q^4/x;q^2\bigr)_n}\\
\phantom{\qquad+}{}=\bigl(q^{\frac{1}{2}};q\bigr)_{\infty}^2
\sum_{m=0}^{\infty}\frac{\bigl(x/q^{\frac{1}{2}},q^{\frac{1}{2}}/x;q\bigr)_m(-q)^m}{(q,-q;q)_m}.
\end{gather*}
Using $q$-Gauss sum \cite[Appendix (II.8)]{GR}
\begin{align*}
\sum_{n=0}^{\infty}\frac{(a,b;q)_n(c/ab)^n}{(q,c;q)_n}
=\frac{(c/a,c/b;q)_{\infty}}{(c,c/ab;q)_{\infty}}
\end{align*}
on the right-hand side of the above identity and simplifying, we have
\begin{gather*}
x^{-1}q^{\frac{1}{2}}\bigl(xq^2,q^2/x;q^2\bigr)_{\infty}R_c\bigl(x,q^2\bigr)
-x^{-1}q\bigl(xq,q/x;q^2\bigr)_{\infty}g_c\bigl(xq,q^2\bigr) \\
\qquad+\bigl(xq,q^3/x;q^2\bigr)_{\infty}R_c\bigl(x/q,q^2\bigr)
-q^{\frac{1}{2}}\bigl(x,q^2/x;q^2\bigr)_{\infty}g_c\bigl(x,q^2\bigr) \\
\phantom{\qquad+}{}=\bigl(q^{\frac{1}{2}};q\bigr)_{\infty}^2\bigl(q;q^2\bigr)_{\infty}^2
\bigl(-xq^{\frac{1}{2}},-q^{\frac{1}{2}}/x;q\bigr)_{\infty}.
\end{gather*}
Utilizing \eqref{1+2}, the above identity can be rewritten as
\begin{gather}
\left(\frac{x^{-1}q^{\frac{3}{2}}\bigl(x,xq^2,1/x,q^2/x;q^2\bigr)_{\infty}}
{\bigl(xq,q/x;q^2\bigr)_{\infty}}
-x^{-1}q\bigl(xq,q/x;q^2\bigr)_{\infty}\right)g_c\bigl(xq,q^2\bigr)\nonumber\\
\qquad+\left(\frac{q\bigl(x/q,xq,q/x,q^3/x;q^2\bigr)_{\infty}}
{\bigl(x,q^2/x;q^2\bigr)_{\infty}}
-q^{\frac{1}{2}}\bigl(x,q^2/x;q^2\bigr)_{\infty}\right)g_c\bigl(x,q^2\bigr)\nonumber\\
\phantom{\qquad+}{}=\bigl(q^{\frac{1}{2}};q\bigr)_{\infty}^2\bigl(q;q^2\bigr)_{\infty}^2
\bigl(-xq^{\frac{1}{2}},-q^{\frac{1}{2}}/x;q\bigr)_{\infty}\nonumber\\
\phantom{\qquad+=}{}-\frac{x^{-1}q^{\frac{1}{2}}\bigl(q;q^2\bigr)_{\infty}^4\bigl(q^2;q^2\bigr)_{\infty}}
{\bigl(xq,q/x;q^2\bigr)_{\infty}^2}
-\frac{\bigl(q;q^2\bigr)_{\infty}^4\bigl(q^2;q^2\bigr)_{\infty}}{\bigl(x,q^2/x;q^2\bigr)_{\infty}^2}. \label{4-4}
\end{gather}
Recall from \cite[equation~(2.4b)]{HM} that
\begin{align*}
j\bigl(-ab;q^2\bigr)j\bigl(-qb/a;q^2\bigr)-aj\bigl(-qab;q^2\bigr)j\bigl(-b/a;q^2\bigr)=j(a;q)j(b;q).
\end{align*}
We put $(a,b)\rightarrow \bigl(q^{\frac{1}{2}},-xq^{\frac{1}{2}}\bigr)$, $\bigl(q^{\frac{1}{2}},-q^{\frac{3}{2}}/x\bigr)$, $\bigl(-q^{\frac{1}{2}}/x,q^{\frac{1}{2}}\bigr)$ to derive
\begin{gather*}
q^{\frac{1}{2}}\bigl(x,xq^2,1/x,q^2/x;q^2\bigr)_{\infty}
-\bigl(xq,q/x;q^2\bigr)_{\infty}^2
=-\frac{j\bigl(q^{\frac{1}{2}};q\bigr)j\bigl(-xq^{\frac{1}{2}};q\bigr)}{J_2^2},\\
q^{\frac{1}{2}}\bigl(x/q,xq,q/x,q^3/x;q^2\bigr)_{\infty}
-\bigl(x,q^2/x;q^2\bigr)_{\infty}^2
=-\frac{j\bigl(q^{\frac{1}{2}};q\bigr)j\bigl(-q^{\frac{3}{2}}/x;q\bigr)}{J_2^2},\\
x^{-1}q^{\frac{1}{2}}\bigl(x,q^2/x;q^2\bigr)_{\infty}^2
+\bigl(xq,q/x;q^2\bigr)_{\infty}^2
=\frac{j\bigl(q^{\frac{1}{2}};q\bigr)j\bigl(-xq^{\frac{1}{2}};q\bigr)}{J_2^2}.
\end{gather*}
Substituting the above three identities into \eqref{4-4}, we have
\begin{gather*}
-\frac{x^{-1}qj\bigl(q^{\frac{1}{2}};q\bigr)j\bigl(-xq^{\frac{1}{2}};q\bigr)}
{\bigl(xq,q/x;q^2\bigr)_{\infty}J_2^2}g_{3,c}\bigl(xq,q^2\bigr)
-\frac{q^{\frac{1}{2}}j\bigl(q^{\frac{1}{2}};q\bigr)j\bigl(-q^{\frac{3}{2}}/x;q\bigr)}
{\bigl(x,q^2/x;q^2\bigr)_{\infty}J_2^2}g_{3,c}\bigl(x,q^2\bigr)\\
\qquad=\bigl(q^{\frac{1}{2}};q\bigr)_{\infty}^2\bigl(q;q^2\bigr)_{\infty}^2
\bigl(-xq^{\frac{1}{2}},-q^{\frac{1}{2}}/x;q\bigr)_{\infty}
-\frac{\bigl(q;q^2\bigr)_{\infty}^4j\bigl(q^{\frac{1}{2}};q\bigr)j\bigl(-xq^{\frac{1}{2}};q\bigr)}
{(x,q/x;q)_{\infty}^2J_2}.
\end{gather*}
The desired result \eqref{5+2} follows by the product rearrangements. Thus we complete the proof.
\end{proof}

\section{Bilateral series of mock theta functions}\label{appendixB}

In this appendix, we present our results regarding bilateral mock theta functions (extracted from the corollaries) in a tabulated list, categorized by the orders of classical mock theta functions. Below the corresponding entries, we also list the relations between these bilateral mock theta functions themselves as well as their relations with classical mock theta functions.
\subsection{Bilateral series of order 2 mock theta functions}
\begin{gather*}
A_c(q):=\sum_{n=-\infty}^{\infty}\frac{\bigl(-q;q^2\bigr)_n q^{(n+1)^2}}{\bigl(q;q^2\bigr)_{n+1}^2},\qquad
B_c(q):=\sum_{n=-\infty}^{\infty}\frac{\bigl(-q^2;q^2\bigr)_n q^{n^2+n}}{\bigl(q;q^2\bigr)_{n+1}^2},\\
\mu_c(q):=\sum_{n=-\infty}^{\infty}\frac{(-1)^n\bigl(q;q^2\bigr)_nq^{n^2}}{\bigl(-q^2;q^2\bigr)_n^2}.
\end{gather*}
Relations
\begin{gather*}
\mu_c(q)+8q\frac{J_4^3}{J_1^3}B_c(q)=8q\frac{J_4^8}{J_1^3J_2^4}
+2\frac{J_4^{20}}{J_1^3J_2^8J_8^8}+\frac{J_1^5}{J_2^4}
-2\frac{J_1J_4^{10}}{J_2^6J_8^4},\\
8q\frac{J_4^3}{J_1^3}B_c(q)+4q\frac{J_4^2}{J_1J_2}\omega(-q)=2\frac{J_4^{20}}{J_1^3J_2^8J_8^8}
+8q\frac{J_4^8}{J_1^3J_2^4}-2\frac{J_1J_4^{10}}{J_2^6J_8^4},\\
\frac{J_2^5}{J_1^3J_4^2}f\bigl(q^2\bigr)+4A_c(-q)=\frac{J_1^5}{J_2^4},\qquad\mu_c(q)-4q\frac{J_4^2}{J_1J_2}\omega(-q)=\frac{J_1^5}{J_2^4}.
\end{gather*}

\subsection{Bilateral series of order 3 mock theta functions}
\begin{alignat*}{3}
& f_c(q):= \sum_{n=-\infty}^{\infty}\frac{q^{n^2}}{(-q;q)_n^2},\qquad&&
 \phi_c(q):= \sum_{n=-\infty}^{\infty}\frac{q^{n^2}}{\bigl(-q^2;q^2\bigr)_n},&\\
&\chi_c(q):= \sum_{n=-\infty}^{\infty}\frac{(-q;q)_nq^{n^2}}{\bigl(-q^3;q^3\bigr)_n},\qquad&&
 \omega_c(q):= \sum_{n=-\infty}^{\infty}\frac{q^{2n^2+2n}}{\bigl(q;q^2\bigr)_{n+1}^2},&\\
&\nu_c(q):= \sum_{n=-\infty}^{\infty}\frac{q^{n^2+n}}{\bigl(-q;q^2\bigr)_{n+1}},\qquad&&
 \rho_c(q):= \sum_{n=-\infty}^{\infty}\frac{\bigl(q;q^2\bigr)_{n+1}q^{2n^2+2n}}{\bigl(q^3;q^6\bigr)_{n+1}},&\\
&\xi_c(q):= 1+2\sum_{n=-\infty}^{\infty}\frac{q^{6n^2-6n+1}}{\bigl(q,q^5;q^6\bigr)_n},\qquad&&
 \sigma_c(q):= \sum_{n=-\infty}^{\infty}\frac{q^{3n^2-3n}}{\bigl(-q,-q^2;q^3\bigr)_n}.&
\end{alignat*}
Relations
\begin{alignat*}{3}
 &\omega_c(q)+\frac{J_1^4J_4^2}{J_2^6}\omega_c(-q)=4\frac{J_1^2J_4^8}{J_2^9}, \qquad&&
 \rho_c(q)+\frac{J_2^3J_3^2J_{12}}{J_1^2J_4J_6^3}\rho_c(-q)=4\frac{J_3J_4^2J_{12}^2}{J_1J_6^3},&\\
& f_c\bigl(q^2\bigr)-4q\frac{J_1^2J_4^4}{J_2^6}\omega_c(-q)
=\frac{J_1^8J_4^2}{J_2^9},\qquad&&
 \phi_c\bigl(q^2\bigr)-2q\frac{J_2J_8^2}{J_4^3}\nu_c\bigl(q^2\bigr)
=\frac{J_1^4J_8}{J_2J_4^3}, &\\
& \chi_c\bigl(q^2\bigr)-q\frac{J_2^3J_3J_{12}^2}{J_1J_4^2J_6^3}\rho_c(-q)
=\frac{J_1^2J_3^2J_{12}}{J_4J_6^3},\qquad&&
 2\frac{J_1^2}{J_2^2}\varPhi(q)=q\omega_c(q),&\\
& \frac{J_2J_3}{J_1J_6}\gamma(q)-3q\sigma_c(q)=\frac{J_1^4}{J_2^2J_3}.&&&
\end{alignat*}

\subsection{Bilateral series of order 6 mock theta functions}
\begin{gather*}
\phi_{-,c}(q):= \sum_{n=-\infty}^{\infty}\frac{(-q;q)_{2n-1}q^n}{\bigl(q;q^2\bigr)_n},\qquad
 \beta_c(q):= \sum_{n=-\infty}^{\infty}\frac{q^{3n^2+3n+1}}{\bigl(q,q^2;q^3\bigr)_{n+1}},\\
\varPhi_c(q):= \sum_{n=-\infty}^{\infty}\frac{(-q;q)_{2n}q^{n+1}}{\bigl(q;q^2\bigr)_{n+1}^2}.
\end{gather*}
Relations
\begin{gather*}
  \sigma_c(q)+\frac{J_2J_3^2}{qJ_1^2J_6}\beta_c(q)=4\frac{J_6^5}{J_1J_2J_3^2},\qquad
 16\varPhi_c\bigl(q^2\bigr)-4q\frac{J_2^5}{J_1^2J_4^3}\omega_c(q)
=2\frac{J_4^{17}}{J_2^8J_8^8}-\frac{J_2^{20}}{J_1^8J_4^{11}},\\
 \beta_c(q)-\frac{J_1}{J_3}\phi_-(q)=q\frac{J_1J_6^6}{J_2^2J_3^4}.
\end{gather*}

\subsection{Bilateral series of order 8 mock theta functions}

\begin{gather*}
U_{0,c}(q):= \sum_{n=-\infty}^{\infty}\frac{\bigl(-q;q^2\bigr)_nq^{n^2}}{\bigl(-q^4;q^4\bigr)_n},\qquad
 U_{1,c}(q):= \sum_{n=-\infty}^{\infty}\frac{\bigl(-q;q^2\bigr)_nq^{(n+1)^2}}
{\bigl(-q^2;q^4\bigr)_{n+1}},\\
V_{1,c}(q):= \sum_{n=-\infty}^{\infty}\frac{\bigl(-q^4;q^4\bigr)_n q^{2n^2+2n+1}}{\bigl(q;q^2\bigr)_{2n+2}}.
\end{gather*}
Relations
\begin{gather*}
 \frac{J_4^2}{J_1J_8}\phi\bigl(q^2\bigr)+2U_{1,c}(-q)=\frac{J_1^3}{J_2J_4},\qquad
 U_{0,c}(-q)-2q\frac{J_2J_8}{J_1J_4}\nu\bigl(q^2\bigr)=\frac{J_1^3}{J_2J_4},\\
 2V_{1,c}(q)-\frac{J_1J_4}{J_2J_8}\psi(-q)=\frac{J_8^9}{J_2J_4^3J_{16}^4}+2q\frac{J_8^3}{J_2J_4}-\frac{J_1^2J_8^4}{J_2^2J_4J_{16}^2}.
\end{gather*}

\subsection{Bilateral series of other mock theta functions}
\begin{gather*}
M_c(q):= \sum_{n=-\infty}^{\infty}\frac{q^{n^2+n}}{(-1,-q;q)_{n+1}},\qquad
N_c(q):=\sum_{n=-\infty}^{\infty}\frac{q^{6n^2+6n+1}}{\bigl(q,q^5;q^6\bigr)_{n+1}},\\
R_{2,c}(q):= \sum_{n=-\infty}^{\infty}\frac{(-1)^n\bigl(q;q^2\bigr)_n(1+q)q^{n^2+2n}}
{\bigl(-q^2;q^2\bigr)_n\bigl(-q^2;q^2\bigr)_{n+1}}.
\end{gather*}
Relations
\begin{gather*}
 M_c\bigl(q^2\bigr)+2q\frac{J_1^2J_4^4}{J_2^6}\omega_c(-q)
=2\frac{J_4^2}{J_2^2}-\frac{J_1^8J_4^2}{2J_2^9},\\
 q^2\sigma_c\bigl(q^2\bigr)-\frac{J_1J_4^2J_6^3}{J_2^3J_3J_{12}^2}N_c(-q)
=\frac{J_4J_6}{J_2J_{12}}-\frac{J_1^2J_3^2J_4}{J_2^3J_{12}},\\
 \beta_c\bigl(q^2\bigr)+\frac{J_2^2J_3}{J_1J_6^2}N_c(q)
=\frac{J_2J_3^6}{J_1^2J_6^4}-\frac{J_2}{J_6},\\
 \mu_c(q)+R_{2,c}(q)=4\frac{J_4^2}{J_1J_2},\qquad
 f_c(q)+2M_c(q)=4\frac{J_2^2}{J_1^2}.
\end{gather*}

\subsection*{Acknowledgements}
The author would like to express sincere gratitude to the anonymous referees, whose insightful comments and constructive suggestions have significantly enhanced the quality and clarity of this paper.
The work is supported by the National Natural Science Foundation of China (Grant No.~12571354), the Shanghai Rising-Star Program (Grant No.~23QA1407300), and the China Scholarship Council (CSC).

\pdfbookmark[1]{References}{ref}
\LastPageEnding

\end{document}